\documentclass[12pt]{article}
\usepackage[utf8]{inputenc}
\usepackage{amsmath,amssymb,amsthm}
\usepackage{hyperref,url}
\usepackage{graphicx}
\usepackage{float}
\usepackage{dsfont}
\usepackage{breqn}
\usepackage{mathtools}
\usepackage{tikz}

\usepackage[top=1in, bottom=1.2in, left=1.2in, right=1.2in]{geometry}

\usepackage{enumitem}
\usepackage{color}

\title{On quasi-isometry invariants associated to the derivation of a Heintze group}

\usepackage{lipsum}

\newcommand{\R}{\mathbb{R}}

\newtheorem{theorem}{Theorem}[section]
\newtheorem{lemma}[theorem]{Lemma}
\newtheorem{corollary}[theorem]{Corollary}

\newenvironment{remark}[1][Remark.]{\begin{trivlist}
\item[\hskip \labelsep \textbf{#1}]}{\end{trivlist}}

\newcommand{\Addresses}{{
\bigskip
\footnotesize

M.~Carrasco Piaggio, \emph{Instituto de Matem\'atica y Estad\'istica, Universidad de la Rep\'ublica}\par\nopagebreak
\textit{E-mail address}, \texttt{mcarrasco@fing.edu.uy}

\medskip

E.~Sequeira, \emph{Centro de Matem\'atica, Universidad de la Rep\'ublica}\par\nopagebreak
\textit{E-mail address}, \texttt{esequeira@cmat.edu.uy}
}}

\begin{document}

\author{Matias Carrasco Piaggio \and Emiliano Sequeira}

\maketitle

\begin{abstract}
A a Heintze group is a Lie group of the form $N\rtimes_\alpha \R$, where $N$ is a simply connected nilpotent Lie group and $\alpha$ is a derivation of $\mathrm{Lie}(N)$ whose eigenvalues all have positive real parts. We show that if two purely real Heintze groups equipped with left-invariant metrics are quasi-isometric, then up to a positive scalar multiple, their respective derivations have the same characteristic polynomial. Using the same thecniques, we prove that if we restrict to the class of Heintze groups for which $N$ is the Heisenberg group, then the Jordan form of $\alpha$, up to positive scalar multiples, is a quasi-isometry invariant.
\end{abstract}
\begin{quote}
\footnotesize{\textbf{Keywords}: quasi-isometry invariant, Heintze group, characteristic polynomial, Jordan form, Heisenberg group}
\end{quote}
\begin{quote}
\footnotesize{\textbf{AMS MSC 2010}: 20F67, 30L10, 53C30}
\end{quote}

\section{Introduction}

Negatively curved homogeneous manifolds where characterized by Heintze in \cite{Heintze74}. Each such manifold is isometric to a solvable Lie group $X_\alpha$ equipped with a left invariant metric, and the group is a semi-direct product $N\rtimes_\alpha \R$ where $N$ is a connected, simply connected, nilpotent Lie group, and $\alpha$ is a derivation of $\mathrm{Lie}(N)$ whose eigenvalues all have positive real parts. Such a group is called a Heintze group.

A purely real Heintze group is a Heintze group $X_\alpha$ as above, for which $\alpha$ has only real eigenvalues. Every Heintze group is quasi-isometric to a purely real one, unique up to isomorphism, see \cite[Section 5B]{Cornulier13}. We will focus only on purely real Heintze groups.

An important conjecture regarding the large scale geometry of purely real Heintze groups states that two such groups are quasi-isometric if, and only if, they are isomorphic \cite[Conjecture 6.B.2]{Cornulier13}. This question is motivated by the original works of Pansu and Hamenst\"adt \cite{UrsulaHamenstadt87,PierrePansu89m,PierrePansu89d}. The conjecture is known to be true in some particular cases, but is still open in its full generality. It was proved by Pansu when both groups are of Carnot type. Recall that a purely real Heintze group is of Carnot type if the eigenspace associated to the smallest eigenvalue of $\alpha$ generates $\mathrm{Lie}(N)$.

If we restrict to the family of abelian type Heintze groups, i.e.\ when $N$ is Abelian, Xie showed that the canonical Jordan form of the derivation $\alpha$, up to scalar multiples, is a quasi-isometry invariant \cite{NageswariXie12,XiangdongXie12,XiangdongXie14}. This is enough to prove the conjecture in the abelian case. In \cite{XiangdongXie14c}, Xie worked out the case when $N$ is the Heisenberg group and $\alpha$ is a diagonalizable derivation. We refer the reader to \cite{Cornulier13} for a more detailed survey on the quasi-isometric classification of locally compact groups, including Heintze groups.

In this note we focus on the quasi-isometry invariants associated to the derivation $\alpha$. Our main motivation is the following question: let $X_\alpha=N_1\rtimes_\alpha \R$ and $X_\beta=N_2\rtimes_\beta \R$ be two quasi-isometric purely real Heintze groups, have the derivations $\alpha$ and $\beta$ the same (up to scalar multiples) Jordan form? In this generality, the question is far from being answered, even if $N_1$ and $N_2$ are the same group.

The main tools to define quasi-isometry invariants for this class of groups are the $L^p$-cohomology and its related cousins, like the Orlicz cohomology, the space of functions of finite $p$-variation, and other similar invariant functional spaces defined on the boundary of the group. For $p\in[1,\infty)$, the local continuous $L^p$-cohomology of $X_\alpha$ can be identified with the Fr\'echet algebra $A_p(N,\alpha)$ of measurable functions $u:N\to \R$ (up to a.e.\ constants) such that the discrete derivative $\Delta u(x,y)=u(x)-u(y)$ has finite $p$-norm on any compact set of $N\times N$. The $p$-norm is with respect to the measure on $N\times N$ given by $dx\otimes dy/\varrho_\alpha(x,y)^{2\mathrm{tr}(\alpha)}$, with $\varrho_\alpha$ an Ahlfors regular parabolic visual quasi-metric on $N$. See Section \ref{preliminaries} for the definition of the parabolic visual boundary.

Since the local continuous $L^p$-cohomology is invariant by quasi-isometries, the function $s_{N,\alpha}:[1,\infty)\to \R$ which associates to $p$ the dimension of the spectrum of $A_p(N,\alpha)$ is also an invariant. This function is locally constant on the complement of the finite set $\left\{\mathrm{tr}(\alpha)/\lambda:\lambda\text{ is an eigenvalue of }\alpha\right\}$, and can have jumps at these points. In many concrete examples, one may use this fact to show that the eigenvalues of $\alpha$ are invariants. Nevertheless, in the general case, it provides not enough information to deal with the entire spectrum of $\alpha$.

Let us explain this point by an example. Consider a finite simple directed graph $\Gamma$, with set of vertices $V=\{v_1,\ldots,v_p\}$ and set of edges $E=\{e_1,\ldots,e_q\}$. We associate to $\Gamma$ a $2$-step nilpotent Lie algebra of dimension $n = p + q$,
$$\mathfrak{n}_\Gamma=\mathrm{Span}(X_1,\ldots,X_p,Z_1,\ldots,Z_q),$$
defined by
$$[X_i,X_j]=\begin{cases}Z_k & \text{ if }e_k=(v_i,v_j),\\ 0&\text{ otherwise}.\end{cases}$$
Two such Lie algebras are isomorphic if, and only if, the graphs from which they arise are isomorphic \cite{Mainkar15}.

Let $\Gamma_1$ and $\Gamma_2$ be the two graphs shown in Figure \ref{graphs1}, $\mathfrak{n}_1$ and $\mathfrak{n}_2$ be the corresponding Lie algebras, and $N_1$ and $N_2$ be the corresponding Lie groups. Define $\alpha$ to be the derivation of $\mathfrak{n}_1$ given by
$$\alpha(X_1)=X_1,\ \alpha(X_2)=2X_2,\ \alpha(X_3)=3X_3,\ \alpha(Z_1)=3Z_1,\ \alpha(Z_2)=5Z_2,\ \alpha(Z_3)=4Z_3;$$
and $\beta$ to be the derivation of $\mathfrak{n}_2$ given by
$$\beta(X_1)=X_1,\ \beta(X_2)=2X_2,\ \beta(X_3)=3X_3,\ \beta(X_4)=3X_4,\ \beta(Z_1)=3Z_1,\ \beta(Z_2)=6Z_2.$$
Consider the associated Heintze groups $N_1\rtimes_\alpha\R$ and $N_2\rtimes_\beta\R$. As far as we know, even though $\alpha$ and $\beta$ are diagonal derivations, there is no previous known result that can be applied to show that these groups are not quasi-isometric. For instance, for all $p\in [1,\infty)$, the dimension of the spectrum of the local continuous $L^p$-cohomology is the same for both groups; see Figure \ref{graphs2}. That is, the $L^p$-cohomology is not able to capture all the eigenvalues of the respective derivations. We refer the reader to \cite[Corollaire 4.4]{Pansu89t} and \cite[Theorem 1.4]{Carrasco14} for the computation of the spectrum of the $L^p$-cohomology.

\begin{figure}[h]
\centering
\setlength{\unitlength}{1cm}
\begin{picture}(5,2.3)
\includegraphics[width=5cm]{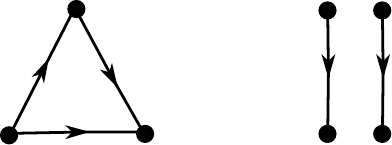}
\put(-5.5,2){\footnotesize (i)}
\put(-2,2){\footnotesize (ii)}
\put(-1.35,1.7){\footnotesize $X_1$}
\put(-1.35,0.0){\footnotesize $X_2$}
\put(-1.25,0.85){\footnotesize $Z_1$}
\put(0,0.85){\footnotesize $Z_2$}
\put(0.05,1.7){\footnotesize $X_3$}
\put(0.05,0.0){\footnotesize $X_4$}

\put(-3.85,1.7){\footnotesize $X_1$}
\put(-3,0.0){\footnotesize $X_2$}
\put(-5.4,0.0){\footnotesize $X_3$}
\put(-3.45,0.85){\footnotesize $Z_1$}
\put(-5,0.85){\footnotesize $Z_3$}
\put(-4.2,0.3){\footnotesize $Z_2$}
\end{picture}
\caption{In (i) the graph $\Gamma_1$ and in (ii) the graph $\Gamma_2$.\label{graphs1}}
\end{figure}
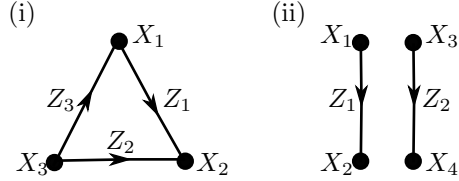

\begin{figure}[h]
\centering
\setlength{\unitlength}{1cm}
\begin{picture}(8,1.5)
\includegraphics[width=8cm]{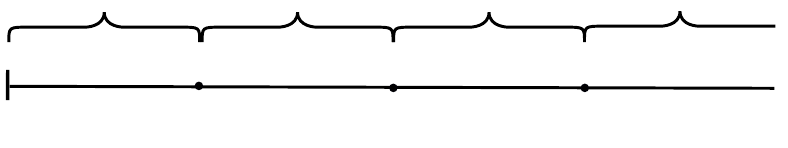}
\put(-8,0.1){\footnotesize $1$}
\put(0.1,0.5){\footnotesize $p$}
\put(-7.5,1.5){\footnotesize $\mathrm{dim}=0$}
\put(-5.5,1.5){\footnotesize $\mathrm{dim}=3$}
\put(-1.6,1.5){\footnotesize $\mathrm{dim}=6$}
\put(-3.6,1.5){\footnotesize $\mathrm{dim}=5$}
\put(-2.5,0.1){\footnotesize $\mathrm{tr}(\alpha)=\mathrm{tr}(\beta)=18$}
\put(-4.7,0.1){\footnotesize $\mathrm{tr}(\alpha)/2=9$}
\put(-6.7,0.1){\footnotesize $\mathrm{tr}(\alpha)/3=6$}
\end{picture}
\caption{Dependence on $p$ of the dimension of the spectrum of the local continuous $L^p$-cohomology for the groups $N_1\rtimes_\alpha\R$ and $N_2\rtimes_\beta\R$ in the example above.\label{graphs2}}
\end{figure}
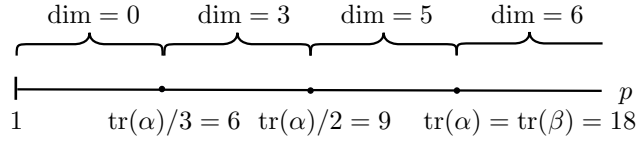

Let us state our main results. We denote by $0<\lambda_1<\cdots <\lambda_d$ the eigenvalues of $\alpha$, and let $\mathfrak{n}=V_1\oplus\cdots\oplus V_d$ be the direct-sum decomposition of $\mathfrak{n}$ into the generalized eigenspaces of $\alpha$. Here $\mathfrak{n}$ denotes the Lie algebra of $N$. Our first result says that the data consisting of $d$, $\lambda_i/\lambda_1$, and $\mathrm{dim}V_i$ for all $i=1,\ldots,d$, is a quasi-isometry invariant of the group.

\begin{theorem}\label{maintheorem}
Let $X_\alpha=N_1\rtimes_\alpha \R$ and $X_\beta=N_2\rtimes_\beta \R$ be two quasi-isometric purely real Heintze groups. Then there exists $s>0$ such that $\alpha$ and $s\beta$ have the same characteristic polynomial.
\end{theorem}

The next corollary follows immediately.

\begin{corollary}
Under the same hypothesis of Theorem \ref{maintheorem}, if $\alpha$ and $\beta$ are diagonalizable derivations, then their respective diagonal forms must be proportional.
\end{corollary}

The proof of Theorem \ref{maintheorem} relies on two important results. The first, is that the boundary extension of any quasi-isometry between two purely real Heintze groups, that are not of Carnot type, induces a homeomorphism between $N_1$ and $N_2$ which preserves a certain foliation \cite[Theorem 1.1]{Carrasco14}. This foliation is given by the left cosets of a proper subgroup of $N_1$; respectively $N_2$. The second, is that such a homeomorphism is necessarily a bi-Lipschitz map when $N_1$ and $N_2$ are endowed with their respective parabolic visual metrics \cite[Theorem 1.1]{LeDonneXie14}. This fact implies that a nicer foliation, given by the left cosets of normal subgroups, is preserved. This allows us to argue by induction.

Our second result shows that the Jordan form of the derivation is a quasi-isometry invariant if we restrict to the class of Heintze groups for which $N$ is the Heisenberg group. 

\begin{theorem}\label{heisenberg}
Let $\alpha$ and $\beta$ be two derivations of the $2n+1$-dimensional Heisenberg algebra $\mathfrak{k}_n$. Denote by $K_n$ the corresponding Heisenberg group. Let $K_n\rtimes_{\alpha}\mathbb{R}$ and $K_n\rtimes_{\beta}\mathbb{R}$ be two purely real quasi-isometric Heintze groups, then there exists $s>0$ such that $\alpha$ and $s\beta$ have the same Jordan canonical form. 
\end{theorem}

This result improves the work done in \cite{XiangdongXie14c,XiangdongXie14b}. The key point in the proof of Theorem \ref{heisenberg} is that the relative simplicity of the algebraic structure of the Heisenberg group, allows us to explicitly compute the normal subgroups that define the invariant foliations. We can then perform a finer induction argument.

\section{Preliminaries on Heintze groups and their quasi-isometries\label{preliminaries}}

In this section we introduce the notations used throughout this note and recall some known results which are used in the proofs of Theorems \ref{maintheorem} and \ref{heisenberg}. 

Let $\mathrm{Der}(\mathfrak{n})$ be the Lie algebra of derivations of $\mathfrak{n}$. The group structure of $X_\alpha$ is given by an expanding action $\tau:\R\to \text{Aut}(N)$, where $\tau$ satisfies $d_e\tau(t)=e^{t\alpha}$. We will identify the subgroup $N\times \{0\}$ with $N$. Let $\mathcal{B}=\{\partial_1,\ldots,\partial_n\}$ be a basis of $T_eN$ on which $\alpha$ assumes its Jordan canonical form, and denote by $\partial_t$ the tangent vector at $s=0$ to the vertical curve $s\mapsto (e,s)$. By \cite{Heintze74}, $X_\alpha$ admits a left-invariant metric of negative curvature. In particular, $X_\alpha$ is Gromov hyperbolic when equipped with any left-invariant metric. We will restrict our attention to metrics of the form
\begin{equation}\label{fixmetric}g_{(x,t)}=\tau(t)^*g_{N,x}\oplus dt^2,\end{equation}
where $g_N$ is a left invariant metric on $N$. The vertical lines $\gamma_x:t\mapsto (x,t)$ are unit-speed geodesics, and they define when $t\to+\infty$ a boundary point denoted $\infty$. The boundary at infinity of $X_{\alpha}$ is a topological $n$-sphere, which we identify with the one-point compactification $N\cup\{\infty\}$.

In $N=\partial X_{\alpha}\setminus \{\infty\}$, we can define a bi-Lipschitz class of quasi-metrics, called parabolic visual metrics, which are characterized by the following properties:
\begin{enumerate}[label={(V}\arabic*{)},leftmargin=1.5cm]
\item $\varrho:N\times N\to \R_+$ is continuous;
\item $\varrho$ is left invariant;
\item and $\tau(t)^*\varrho=e^{t}\varrho$ for all $t\in\R$. 
\end{enumerate}
Notice that if a quasi-metric $\varrho$ satisfies (V1), (V2) and (V3) above, then for any $x_0\in N$ the function $x\mapsto \varrho(x_0,x)$ is proper. 

A particular choice of visual metric is the following. The orbits of $N$, the sets $N\times\{t\}$, correspond to the horospheres centered at $\infty$. We denote by $d_t$ the Riemannian distance induced on $N\times \{t\}$. Then
\begin{equation}\label{quasimetric}
\varrho(x,y)= e^{t},\text{ where }t=\inf\left\{s\in\R:d_s\left(\gamma_x(s),\gamma_y(s)\right)\leq 1\right\},\ x,y\in N,
\end{equation} 
defines a visual metric on $N$. We refer to \cite{Hersonsky-Paulin97,Hamenstadt89} for more details. Notice that $(N,\varrho)$ is Ahlfors regular of Hausdorff dimension $\mathrm{tr}(\alpha)$. When necessary, we will write $\varrho_{\alpha}$ to indicate the dependence of the visual metric on the derivation $\alpha$. Notice that the bi-Lipschitz class of the visual metrics is unchanged when the metric $g_N$ of $N$ is changed. This is not the case for changes involving the vertical direction $\partial_t$.

Denote by $0<\lambda_1<\cdots<\lambda_d$ the eigenvalues of $\alpha$, and let $\mathfrak{n}=V_1\oplus\cdots\oplus V_d$ be the direct sum decomposition of $\mathfrak{n}$ into the generalized eigenspaces of $\alpha$. One easily checks that
\begin{equation}\label{ursula}
[V_i,V_j]\text{ is }\begin{cases}\text{ contained in } V_k \text{ if }\lambda_k=\lambda_i+\lambda_j,\\ \text{ trivial otherwise.}\end{cases}
\end{equation}
In particular, $V_d$ is always contained in the center of $\mathfrak{n}$. This also implies that if $\alpha=\delta+\nu$ is the decomposition of $\alpha$ into its diagonal and nilpotent parts, then both $\delta$ and $\nu$ are derivations of $\mathfrak{n}$.

A special role in this theory is played by the purely real Heintze groups of Carnot type. This is the case when the Lie algebra spanned by the eigenspace $W_1=\mathrm{Ker}(\alpha-\lambda_1)$ coincides with $\mathfrak{n}$. More precisely, if we let $W_{i+1}=[W_1,W_i]$, then these subspaces form a grading of $\mathfrak{n}$; that is
$$\mathfrak{n}=\bigoplus_{i=1}^dW_i\text{ and }[W_i,W_j]\subset W_{i+j},$$
and we have $\alpha=i\lambda_1$ on $W_i$. Thus $\alpha$ is diagonalizable, its eigenvalues are $\lambda_1,2\lambda_1,\ldots,d\lambda_1$, and $W_1,\ldots, W_d$ are its eigenspaces. Any norm on $W_1$ induces a Carnot-Carath\'eodory metric $\varrho_{CC}$ on $N$ that satisfies properties (V1), (V2) and (V3) above. In particular, this distance makes $N$ a geodesic space. Examples of Carnot type groups are the symmetric spaces of rank one.

When $X_\alpha$ is not of Carnot type, we denote by $m_{ir}$, for $i=1,\ldots,d$ and $r=1,\ldots,n_i$, the sizes of the Jordan blocks corresponding to the generalized eigenspace $V_i$, and let $m_i=\max\{m_{ir}:r=1,\ldots,n_i\}$. We consider the Lie algebras 
\begin{equation}\label{algebrah}\mathfrak{u}_\alpha=\mathrm{LieSpan}(V_1)\text{ and }\mathfrak{h}_{\alpha}=\mathrm{LieSpan}(V_1^*),\end{equation}
where $V_1^*$ is the vector space spanned by the $\lambda_1$-eigenvectors corresponding to Jordan blocks of maximal size $m_1$. Notice that $\mathfrak{h}_\alpha$ is a proper sub-algebra of $\mathfrak{n}$. Finally, let $U_\alpha$ and $H_{\alpha}$ be the corresponding connected Lie subgroups of $N$ whose Lie algebras are $\mathfrak{u}_\alpha$ and $\mathfrak{h}_{\alpha}$ respectively.

Suppose $X_\alpha=N_1\rtimes_\alpha\R$ and $X_\beta=N_2\rtimes_\beta\R$ are two quasi-isometric purely real Heintze groups. Then there exists a quasi-isometry between them whose boundary extension sends $\infty_1$ to $\infty_2$ \cite[Lemma 6.D.1]{Cornulier13}. In particular, there exists a quasi-symmetric map between $(N_1,\varrho_{\alpha})$ and $(N_2,\varrho_{\beta})$. Conversely, any such quasi-symmetric map induces a quasi-isometry between $X_\alpha$ and $X_\beta$ \cite{BonkSchramm00,FredericPaulin96}.

\medskip
\noindent Let us end this section by outlining the following known results which will be important in the proofs of Theorems \ref{maintheorem} and \ref{heisenberg}:
\begin{enumerate}[label={(Fact }\arabic*{)},leftmargin=1.5cm]
\item A Carnot type Heintze group cannot be quasi-isometric to a non-Carnot type one, see \cite[Corollary 1.9]{Carrasco14}.
\item If two Carnot type Heintze groups are quasi-isometric, then they are isomorphic, see \cite{PierrePansu89m}.
\item If $\Phi:(N_1,\varrho_1)\to(N_2,\varrho_2)$ is a quasisymmetric map between the boundaries of non-Carnot type Heintze groups, then $\Phi$ sends the left cosets of $H_\alpha$ (resp. $U_\alpha$) into the left cosets of $H_\beta$ (resp. $U_\beta$), see \cite[Theorem 1.4]{Carrasco14}.
\item If $\Phi$ is like in (Fact 3) and the smallest eigenvalue of $\alpha$ and $\beta$ coincide, then $\Phi$ is a bi-Lipschitz map, see \cite[Theorem 1.1]{LeDonneXie14} and \cite[Corollary 1.8]{Carrasco14}.
\end{enumerate}
This last point allows us to use bi-Lipschitz invariants to study quasi-isometries between non-Carnot type Heintze groups. It is a consequence of (Fact 3) and a general theorem proved by Le Donne and Xie on rigidity of fibre-preserving quasisymmetric maps \cite{LeDonneXie14}.

\section{Hausdorff distance between left cosets}

In this section we fix a parabolic visual metric $\varrho_{\alpha}$ on $N$, and let $\mathfrak{h}$ be a proper sub-algebra of $\mathfrak{n}$. The goal is to determine which are the lefts cosets of $H$, the connected subgroup of $N$ whose Lie algebra is $\mathfrak{h}$, that are at finite Hausdorff distance from $H$.

The normalizer of $\mathfrak{h}$ is the subalgebra $\mathfrak{N}(\mathfrak{h})=\{X\in\mathfrak{n}:[X,\mathfrak{h}]\subset \mathfrak{h}\}$. We denote by $N(H)$ the corresponding normalizer of $H$ in $N$. Also, recall that the Hausdorff distance between two subsets of $N$ is given by
$$\mathrm{dist}(A,B)=\max\left\{\sup_{a\in A}\mathrm{dist}(a,B),\sup_{b\in B}\mathrm{dist}(b,A)\right\},$$
where $\mathrm{dist}(a,B)=\inf_{b\in B}\varrho_{\alpha}(a,b)$. We first prove the following estimate for the visual metric $\varrho_{\alpha}$.

\begin{lemma}\label{distanceorigin}
Let $\langle \cdot,\cdot\rangle$ be an inner product on $\mathfrak{n}$. For any $\mu>\lambda_d$, there exists a constant $c>0$ such that
$$c\,\|X\|\leq \varrho(e,\exp{X})^{\mu}$$
for all $X\in \mathfrak{n}$ with $\varrho_{\alpha}(e,\exp{X})\geq 1$.
\end{lemma}

\begin{proof} Let us define $C=\max\{\|X\|: X\in\mathfrak{n}, \varrho_{\alpha}(e,\exp{X})=1\}$, which is finite because $\{x\in N: \varrho_{\alpha}(e,x)=1\}$ is a compact set. Fix $X\in\mathfrak{n}$ with $\varrho_{\alpha}(e,\exp{X})\geq 1$, and let $x=\exp{X}$. We choose $t\in\mathbb{R}$ such that 
$$\varrho_{\alpha}(e,\tau(t)x)=e^t\varrho_{\alpha}(e,x)=1.$$
Notice that $t\leq 0$. Since $\tau(t)x=\exp(e^{t\alpha}X)$, we have $\|e^{t\alpha}X\|\leq C$. Also, the following lower bound holds:
$$\|e^{t\alpha}X\|\geq \frac{e^{t\lambda_d}\|X\|}{p(t)},$$
where $p(s)\sim c_1|s|^{n_d-1}$ when $s\to-\infty$, $c_1>0$ is a constant, and $n_d$ is the dimension of the generalized eigenspace associated to $\lambda_d$. Given $\mu>\lambda_d$, we can choose a constant $c_2>0$ such that
$$\frac{e^{s\lambda_d}}{p(s)}\geq c_2 e^{s\mu}\ \text{ for all } s\leq 0.$$
From the above inequalities, we get $c_2 e^{t\mu}\|X\|\leq C$, and then
$$\frac{c_2}{C}\|X\| \leq \varrho_{\alpha}(e,\exp{X})^{\mu}.$$   
This finishes the proof.
\end{proof}

Suppose that $K$ is a connected Lie subgroup of $N$, and let $\mathfrak{k}$ be its Lie algebra. Since $\varrho_{\alpha}$ is left invariant, for any $x\in N$, we have 
$$\mathrm{dist}_H(Kx,K)=\mathrm{dist}(x,K).$$
In particular, if $x\in N(K)$, then 
\begin{equation}\label{direct}
\mathrm{dist}_H(xK,K)=\mathrm{dist}_H(Kx,K)=\mathrm{dist}(x,K)<\infty.
\end{equation}
This shows that if $x,y\in N(K)$, the the left cosets $xK$ and $yK$ are at finite Hausdorff distance. Moreover, (\ref{direct}) implies that $xK$ and $yK$ are parallel in the following sense:
\begin{equation}\label{parallels}\mathrm{dist}_H(xK,yK)=\mathrm{dist}(xk,yK)=\mathrm{dist}(xK,yk)\end{equation}
for all $k\in K$.

Suppose in addition that $K$ is a $\tau$-invariant normal subgroup. We can define using (\ref{parallels}) a quasi-metric $\tilde{\varrho}_\alpha$ on the quotient $N/K$, and with this quasi-metric, the canonical projection $\pi:N\to N/K$ is a $1$-Lipschitz map. Consider $\tilde\tau:\R\to \mathrm{Aut}(N/K)$ the morphism induced on the quotient by $\tau$. We identify the Lie algebra $\mathrm{Lie}(N/K)$ with $\mathfrak{n}/\mathfrak{k}$ via the isomorphism $d_e\pi$. With this identification, the action $\tilde{\tau}$ is generated by the derivation
$$\tilde{\alpha}(X+\mathfrak{k})=\alpha(X)+\mathfrak{k},\ \text{ for }X\in \mathfrak{n}.$$
Denote by $p_{\alpha}$ and $p_{\tilde{\alpha}}$ the characteristic polynomials of $\alpha$ and $\tilde{\alpha}$. Since $p_{\tilde{\alpha}}$ divides $p_{\alpha}$, the eigenvalues of $\tilde{\alpha}$ are all positive.

One can check that the quotient quasi-metric $\tilde\varrho_\alpha$ satisfies the conditions (V1), (V2) and (V3) of Section \ref{preliminaries}. In particular, $\tilde\varrho_\alpha$ is in the same bi-Lipschitz class as the parabolic visual metrics on the boundary of $N/K\rtimes_{\tilde{\alpha}}\R$.

\begin{lemma}\label{distancecosets}
Let $H$ be a Lie subgroup of $N$. Two left cosets of $H$ are at finite Hausdorff distance if, and only if, they are contained in the same left coset of $N(H)$. Moreover, in that case they are parallel in the sense of (\ref{parallels}).
\end{lemma}

This lemma is proved for Carnot groups with Carnot-Carath\'eodory metrics in \cite[Lemma 4.4]{LeDonneXie14}. 

\begin{proof}
The direct implication was already proved in (\ref{direct}), so let us prove the converse. We will argue by induction on the nilpotency of $N$. The Abelian case is obvious because $N(H)=N$. Let us assume the claim true for nilpotent groups of nilpotency less than $m$.

Let $Z$ be the center of $N$. Then the quotient group $N/Z$ has nilpotency at most $(m-1)$. Let us take two left cosets of $H$, contained in different left cosets of $N(H)$. Since $\varrho_{\alpha}$ is left-invariant, we may assume that they are $H$ and $xH$ with $x\notin N(H)$. This means that there is $h_0\in H$ such that $x^{-1}h_0x\notin H$. Let $\pi:N\to N/Z$ be the canonical projection, there are two posibilities:

\begin{enumerate}
\item $\pi(x)\notin N(\pi(H))$; or
\item $\pi(x)\in N(\pi(H))$ and therefore $\pi(x^{-1}hx)\in\pi(H)$ for all $h\in H$.
\end{enumerate}

We apply the induction hypothesis in the first case: $\pi(x)\pi(H)=\pi(xH)$ and $\pi(H)$ must be at infinite Hausdorff distance. Since $\pi$ is a contraction, the Hausdorff distance between the cosets $H$ and $xH$ is also infinite.

Suppose the second case holds. Let $X,Y_0\in \mathfrak{n}$ be such that $\exp{X}=x$ and $\exp{Y_0}=h_0$, and define $h_0^{(t)}:=\exp(tY_0)$ for $t\in\R$. Notice that $h_0^{(t)}$ is also in $H$. For each $t\in\mathbb{R}$, there exists $h_t\in H$ such that $h_t^{-1}x^{-1}h_0^{(t)}x\in Z$. This is because $\pi(x^{-1}h_0^{(t)}x)\in\pi(H)$. In particular, $x^{-1}h_0^{(t)}x\in N(H)$.

By the Baker-Campbell-Hausdorff formula, since
$$x^{-1}h_0x=\exp \left( \sum_{j=0}^{\infty}\frac{\mathrm{ad}_{X}^j(Y_0)}{j!} \right)\notin H,$$
we then have that
$$W:=\sum_{j=0}^{\infty}\frac{\mathrm{ad}_{X}^j(Y_0)}{j!}-Y_0 \notin \mathfrak{h}.$$
Also, notice that for any $t$ we have
$$x^{-1}h_0^{(t)}x=\exp \left( \sum_{j=0}^{\infty}\frac{\mathrm{ad}_{X}^j(tY_0)}{j!}\right)=\exp(tY_0+tW).$$
Let $h$ be any point in $H$, and consider $Y\in \mathfrak{h}$ such that $\exp{Y}=h$. Then, since $tY_0+tW$ is in $\mathfrak{N}(\mathfrak{h})$, we can write
$$h^{-1}x^{-1}h_0^{(t)}x=\exp(Y(t)+tW)$$
with $Y(t)\in\mathfrak{h}$ and $tW\notin\mathfrak{h}$.

Let us endow $\mathfrak{n}$ with an inner product making $W$ orthogonal to $\mathfrak{h}$. Then 
$$\|Y(t)+tW\|^2=\|Y(t)\|^2+\|tW\|^2\geq t^2\|W\|^2.$$ 
Since $W$ is a fixed non trivial vector, we can find a big enough uniform $t$, such that $\varrho_{\alpha}(\exp(Y(t)+tW),e)\geq 1$. We fix any $\mu>\lambda_d$, then by Lemma \ref{distanceorigin}, there exists $c>0$ such that
\begin{align*}
\inf_{h\in H} \varrho_{\alpha}(h,x^{-1}h_0^{(t)}x)&\geq \inf_{h\in H}c\|Y(t)+tW\|^{\frac{1}{\mu}} \geq c(t\|W\|)^{\frac{1}{\mu}}.
\end{align*}
This implies that
$$\mathrm{dist}_{H}(xH,Hx)=\mathrm{dist}_{H}(H,x^{-1}Hx)\geq \sup_{t\geq 0} c(t\|W\|)^{\frac{1}{\mu}}=\infty.$$
Recall that by (\ref{direct}), $H$ and $Hx$ are always at finite Hausdorff distance. We then conclude that $\mathrm{dist}_H(xH,H)=\infty$.
\end{proof}

\section{Proof of Theorem \ref{maintheorem}}

By (Fact 1) and (Fact 2), we can suppose that $X_\alpha$ and $X_\beta$ are not of Carnot type. We may also assume, after multiplying by a positive number, that $\alpha$ and $\beta$ have the same smallest eigenvalue. Let $F:X_\alpha\to X_\beta$ be a quasi-isometry, and let $\Phi$ be its boundary extension. By \cite[Lemma 6.D.1]{Cornulier13}, we may assume that $\Phi$ sends $\infty_1$ to $\infty_2$. By (Fact 4), we know that $\Phi:(N_1,\varrho_{\alpha})\to (N_2,\varrho_{\beta})$ is a bi-Lipschitz homeomorphism. Theorem \ref{maintheorem} follows from the next lemma.

\begin{lemma}\label{premaintheorem}
Let $N_1$ and $N_2$ be two simply connected nilpotent Lie groups and $\alpha\in \mathrm{Der}(n_1),\ \beta\in \mathrm{Der}(n_2)$ with positive eigenvalues. If $\Phi:(N_1,\varrho_{\alpha})\to (N_2,\varrho_{\beta})$ is a bi-Lipschitz homeomorphism, then $\alpha$ and $\beta$ have the same characteristic polynomial.
\end{lemma}

\begin{proof} Let $\mathfrak{h}_{\alpha}$ and $\mathfrak{h}_{\beta}$ be the Lie sub-algebras defined in (\ref{algebrah}), and let $H_{\alpha}\leq N_1$ and $H_{\beta}\leq N_2$ be the corresponding connected Lie subgroups.

By (Fact 3), the left cosets of $H_{\alpha}$ are mapped by $\Phi$ into the left cosets of $H_{\beta}$. Let us consider the increasing sequence of normalizers
$$H_{\alpha}=H_0^{\alpha}\triangleleft H_1^{\alpha}\triangleleft H_3^{\alpha}\triangleleft...$$
where $H_{i+1}^{\alpha}$ is the normalizer of $H_{i}^{\alpha}$, and 
$$\mathfrak{h}_{\alpha}=\mathfrak{h}_0^{\alpha}\triangleleft \mathfrak{h}_1^{\alpha}\triangleleft \mathfrak{h}_3^{\alpha}\triangleleft...$$
the corresponding Lie algebras. In the same way let $H_i^{\beta}$ and $\mathfrak{h}_i^{\beta}$ be the corresponding sequence for $\beta$. Notice that there exist $k_1$ and $k_2$ such that $H_{k_1}^{\alpha}=N_1$ and $H_{k_2}^{\beta}=N_2$. This is because the normalizer of a proper subgroup of a nilpotent Lie group is always strictly bigger than the subgroup \cite[Lemma 4.2]{LeDonneXie14}.

By Lemma \ref{distancecosets}, the left cosets of $H_j^{\alpha}$ are mapped into the left cosets of $H_j^{\beta}$ for all $j$. This follows because $\Phi$ is bi-Lipschitz. In particular, $k_1=k_2=k$.

We will argue by induction on the dimension $n$ of the groups. The case $n=1$ is trivial. Let us suppose that the lemma is true for dimension strictly smaller than $n$. Composing $\Phi$ with a translation if necessary, we can assume that $H_{k-1}^{\alpha}$ is mapped into $H_{k-1}^{\beta}$. The restriction of the quasi-metric $\varrho_\alpha$ to $H_{k-1}^\alpha$ is bi-Lipschitz equivalent to the quasi-metric induced by the derivation $\alpha|_{\mathfrak{h}_{k-1}^\alpha}$. The same holds for the restriction of $\varrho_\beta$ to $H_{k-1}^\beta$. Since $\dim(H_{k-1}^{\alpha})=\dim(H_{k-1}^{\beta})<n$, then $\alpha|_{\mathfrak{h}_{k-1}^{\alpha}}$ and $\beta|_{\mathfrak{h}_{k-1}^{\beta}}$ have the same characteristic polynomial; i.e.\ $p_{\alpha|_{\mathfrak{h}_{k-1}^{\alpha}}}=p_{\beta|_{\mathfrak{h}_{k-1}^{\beta}}}$.

As we did in Lemma \ref{distancecosets}, one can prove that the quasi-metric $\varrho_{\alpha}$ induces a quasi-metric $\tilde{\varrho}_{\alpha}$ on the quotient $N_1/H_{k-1}^{\alpha}$ bi-Lipschitz equivalent to the quasi-metric induced by $\tilde{\alpha}\in \mathrm{Der}(\mathfrak{n}_1/\mathfrak{h}_{k-1}^{\alpha})$. In the same way is defined $\tilde{\varrho}_{\beta}$.

The homeomorphism $\Phi$ induces a bi-Lipschitz homeomorphism 
$$\tilde{\Phi}:(N_1/H_{k-1}^{\alpha},\tilde{\varrho}_{\alpha})\to (N_2/H_{k-1}^{\beta},\tilde{\varrho}_{\beta}).$$
Then by induction hypothesis, $\tilde{\alpha}$ and $\tilde{\beta}$ have the same characteristic polynomial; i.e. $p_{\tilde{\alpha}}=p_{\tilde{\beta}}$. Putting it all together $p_{\alpha}=p_{\tilde{\alpha}}p_{\alpha|_{\mathfrak{h}_{k-1}^{\alpha}}}=p_{\tilde{\beta}}p_{\beta|_{\mathfrak{h}_{k-1}^{\beta}}}=p_{\beta}.$  
\end{proof}
 
\begin{remark}[Remark 4.1.]
In the proof of Lemma \ref{dimminlambda} of the next section, we will use again the fact that the restriction of a parabolic visual metric to an $\alpha$-invariant subgroup $H$ is bi-Lipschitz equivalent to the visual metrics on $H$ induced by the derivation $\alpha|_{\mathfrak{h}}$.
\end{remark}

\section{Minimal Hausdorff-dimensional curves}

The goal of this section is to give an alternative and more elementary proof of the second part of (Fact 3), i.e.\ that the left cosets of $U_\alpha$ are invariant under bi-Lipschitz maps. This follows from Lemma \ref{dimclass} below. We will use it in Section \ref{SectionHeisenberg} for the proof of Theorem \ref{heisenberg}. It also provides an elementary proof that $\lambda_1$ is invariant. Lemma \ref{dimdiagpart} may be useful in further investigations for defining higher-dimensional bi-Lipschitz invariants, like the minimal Hausdorff dimension of embedded surfaces in $N$.

If $(M,\varrho)$ is a quasi-metric space, the $t$-dimensional Hausdorff measure and the Hausdorff dimension of a subset $A\subset M$ are defined exactly in the same way as for metric spaces. When $\varrho$ is a distance, the Hausdorff dimension of a non-degenerate connected set is always bounded below by $1$. It is well known that when $\varrho$ is just a quasi-metric, there exists $\epsilon>0$ such that $\varrho^\epsilon$ is bi-Lipschitz equivalent a distance in $M$. Moreover, the Hausdorff dimensions of a subset $A$ for the quasi-metrics $\varrho$ and $\varrho^\epsilon$ are related by 
$$\mathrm{Hdim}_{\varrho^\epsilon}(A)=\frac{\mathrm{Hdim}_{\varrho}(A)}{\epsilon}.$$
This implies that $d_0(\varrho)=\inf\left\{\mathrm{Hdim}_\varrho(A):A\text{ is a non-degenerate connected set}\right\}$ is a positive number. We will compute the invariant $d_0$ for the parabolic visual quasi-metrics on $N$ in terms of the smallest eigenvalue of $\alpha$.

\begin{lemma}\label{dimminlambda}
Let $(N,\varrho)$ be the parabolic visual boundary of $X_\alpha$, and let $\lambda_1>0$ be the smallest eigenvalue of $\alpha$. Then $d_0(\varrho)=\lambda_1$.
\end{lemma}

In the proof of Lemma \ref{dimminlambda} we will use another lemma which reduces the computation to the case when $\alpha$ is diagonalizable. We decompose $\alpha=\delta+\nu$ where $\delta$ is diagonalizable and $\nu$ is nilpotent. By (\ref{ursula}), both $\delta$ and $\nu$ are derivations on $\mathfrak{n}$. Consider the Heintze group $X_\delta=N\rtimes_\delta \R$ defined by the diagonal part of $\alpha$. We use the same inner product on $\mathfrak{n}$ to define the left-invariant Riemannian metrics on $X_\alpha$ and $X_\delta$. Let $\varrho_\alpha$ and $\varrho_\delta$ be the parabolic visual quasi-metrics on $N$ induced by $\alpha$ and $\delta$ respectively.

\begin{lemma}\label{dimdiagpart}
For any subset $A$ of $N$, we have $\mathrm{Hdim}_{\varrho_\alpha}(A)=\mathrm{Hdim}_{\varrho_\delta}(A)$.
\end{lemma}
\begin{proof}
We will show that for any $\mu>1$, there exists a constant $C=C(\mu)\geq 1$ such that
\begin{equation}\label{compdiagmetric}\frac{1}{C}\varrho_\delta(x,y)^\mu\leq \varrho_\alpha(x,y)\leq C\varrho_\delta(x,y)^{1/\mu},\end{equation}
for all $x,y\in N$. The constant $C$ usually explodes when $\mu\to 1$.

We first compare the Riemannian metrics on the tangent bundle of $N\times\{t\}$ for $t\leq 0$ induced by $\alpha$ and $\delta$. Since these Riemannian metrics are left-invariant, we can think of them as one-parameter families of inner products on $T_eN$, given by
$$\left\|v\right\|^\delta_t=\left\|e^{-t\delta}v\right\|_0\text{ and }\left\|v\right\|^\alpha_t=\left\|e^{-t\alpha}v\right\|_0\ \ (v\in T_eN),$$
respectively. Fix $v\in T_eN$, and let $v=v_1+\cdots+v_d$ with $v_i\in V_i$, be the decomposition of $v$ as a sum of vectors belonging to the generalized eigenspaces of $\alpha$. Then
\begin{align*}
\left\|e^{-t\alpha}v\right\|_0&=\left\|e^{-t\nu}e^{-t\delta}v\right\|_0\leq p(t)\left\|e^{-t\delta}v\right\|_0,
\end{align*}
where
$$p(t)=K\sum_{j=0}^{m-1}\frac{|t|^j}{j!},$$
and $m$ is the order of nilpotency of $\nu$ and $K=1+\|\nu\|_0$. Let us also fix $\mu>1$. Since the generalized eigenspaces $V_i$ are orthogonal, we have
\begin{align*}
p(t)^2\left\|e^{-t\delta}v\right\|_0^2&=\sum_{i=1}^dp(t)^2e^{-2t\lambda_i}\left\|v_i\right\|_0^2\leq C\sum_{i=1}^de^{-2\mu t\lambda_i}\left\|v_i\right\|_0^2=C\left\|e^{-t\mu\delta}v\right\|_0^2,
\end{align*}
where the inequality holds for all $t\leq 0$ if $C$ is chosen big enough depending only on $\mu$, the data defining $p(t)$, and the eigenvalues $\lambda_i$ for $i=1,\ldots,d$. Therefore $C$ depends only on $ \mu$ and $\alpha$.

To show the reverse inequality, notice that
$$\left\|e^{-t\delta}v\right\|_0=\left\|e^{t\nu}e^{-t\alpha}v\right\|_0\leq p(t)\left\|e^{-t\alpha}v\right\|_0,$$
and therefore
\begin{align*}
\left\|e^{-t\alpha}v\right\|_0&\geq p(t)^{-2}\left\|e^{-t\delta}v\right\|_0^2=\sum_{i=1}^dp(t)^{-2}e^{-2t\lambda_i}\left\|v_i\right\|_0^2\\
& \geq \frac{1}{C}\sum_{i=1}^de^{-2\frac{1}{\mu}t\lambda_i}\left\|v_i\right\|_0^2=\frac{1}{C}\left\|e^{-t\frac{1}{\mu}\delta}v\right\|_0^2.
\end{align*}
Here, as before, the inequality holds for all $t\leq 0$ and the constant $C$ depends only on $\mu$ and the derivation $\alpha$.

In summary, we have shown that for any $\mu>1$, there is a constant $C$ such that for all $t\leq 0$ and $v\in T_eN$ we have
$$\frac{1}{C}\left\|e^{-t\frac{1}{\mu}\delta}v\right\|_0\leq \left\|e^{-t\alpha}v\right\|_0\leq C\left\|e^{-t\mu\delta}v\right\|_0.$$
Since $\varrho_{\mu\delta}=\varrho_{\delta}^{1/\mu}$ and $\varrho_{1/\mu\delta}=\varrho_{\delta}^\mu$, the conclusion (\ref{compdiagmetric}) follows from the last inequalities.
\end{proof}

\begin{proof}[Proof of Lemma \ref{dimminlambda}]
By Lemma \ref{dimdiagpart}, we may assume that $\alpha$ is diagonalizable. We will prove by induction on the nilpotency of $N$ that $d_0\geq \lambda_1$.

First, suppose $N$ is Abelian. Then the parabolic visual quasi-metric is bi-Lipschitz equivalent to the quasi-metric given by
$$\hat{\varrho}(x,y)=\sum_{i=1}^d\left\|x_i-y_i\right\|_0^{\frac{1}{\lambda_i}}.$$
Since $\|\cdot\|_0$ is an Euclidean norm, it follows that $d_0(\varrho)=\lambda_1$ in this case.

Suppose now that the assertion holds for every $k$-nilpotent group, and let $N$ be $(k+1)$-nilpotent with $k\geq 0$. Let $Z$ be the center of $N$ and let $\pi:N\to N/Z$ be the canonical projection. Let $A$ be a non-degenerate connected subset of $N$, then if $\pi(A)$ is a connected subset of $N/Z$. Since $N/Z$ is at most $k$-nilpotent, if $\pi(A)$ is non-degenerate we have by induction hypothesis that
$$\lambda_1\leq \mathrm{Hdim}_{\tilde{\varrho}}(\pi(A))\leq\mathrm{Hdim}_{\varrho}(A),$$
where $\tilde{\varrho}$ is the parabolic visual metric on $N/Z$, and the second inequality follows since $\pi$ is a $1$-Lipschitz map for this quasi-metric.

Suppose that $\pi(A)$ consist of a singleton, so $A$ is contained in a left coset $xZ$ for some $x\in N$. Recall that the center $\mathfrak{z}$ of $\mathfrak{n}$ is $\alpha$-invariant, so we can restrict the action of $\alpha$ to $\mathfrak{z}$. Moreover, the eigenvalues of the restriction of $\alpha$ to $\mathfrak{z}$ are bounded below by $\lambda_1$. Since $Z$ is Abelian, by the induction hypothesis we have that $\mathrm{Hdim}(A)\geq \lambda_1$. This concludes the proof of the claim.

We now show the reverse inequality. Let $\mathfrak{u}_{\alpha}=\mathrm{Lie}(V_1)$ and $U_{\alpha}\leq N$ its associated subgroup. By definition, $U_{\alpha}\rtimes_\alpha\R$ is a Carnot type Heintze group, and therefore, the restriction of the parabolic visual quasi-metric $\varrho$ to $U_{\alpha}$ is bi-Lipschitz equivalent to the snow-flake $\varrho_{\mathrm{CC}}^{1/\lambda_1}$, where $\varrho_{\mathrm{CC}}$ is the Carnot-Carath\'eodory distance on $U$. In particular, $\varrho_{\mathrm{CC}}$ is a geodesic distance, so there are many one-dimensional curves. From this it follows that $d_0(\varrho)\leq \lambda_1$.
\end{proof}

Let $x$ and $y$ be two points in $N$, then define $x\sim_1 y$ if they are equal or if there exists a connected set $A$ with Hausdorff dimension $\lambda_1$ containing both points. It is clear that any left coset of $U_{\alpha}$ is contained in an equivalence class of $\sim_1$.

\begin{lemma}\label{dimclass}
The equivalence classes of $\sim_1$ coincide with the left cosets of $U_{\alpha}$.
\end{lemma}
\begin{proof}
It remains to prove that an equivalence class of $\sim_1$ is contained in a left coset of $U_{\alpha}$. Consider the increasing chain of normalizers
$$U_{\alpha}=N_0\triangleleft N_1\triangleleft N_2\triangleleft\cdots,$$ 
where $N_{j+1}$ is the normalizer of $N_{j}$. Let $k$ be such that $N_{k}=N$.

Let $x,y\in N$ be two different points contained in a connected set $A$ of Hausdorff dimension $\lambda_1$. Denote by $\pi:N\to N/N_{k-1}$ the canonical projection. Then since $N_{k-1}$ contains $U$, the eigenvalues of the derivation $\tilde{\alpha}$, induced by $\alpha$ on $\mathfrak{n}/\mathfrak{n}_{k-1}$, are all strictly bigger than $\lambda_1$. If $\pi(A)$ is non-degenerate, then by Lemma \ref{dimminlambda}, we have
$$\lambda_1<\mathrm{Hdim}(\pi(A))\leq \mathrm{Hdim}(A)=\lambda_1,$$
which is a contradiction. Therefore $\pi(A)$ is a singleton, and $A$ is contained in a left coset of $N_{k-1}$. Arguing by induction we deduce that $A$ is contained in a left coset of $U_{\alpha}$.
\end{proof}

\section{Proof of Theorem \ref{heisenberg}\label{SectionHeisenberg}}

Let $K_n$ be the Heisenberg group of dimension $2n+1$, $\mathfrak{k}_n$ be its corresponding Lie algebra and $\mathfrak{z}$ be the center of $\mathfrak{k}_n$. Let $B=\{X_1,...X_n,Y_1,...,Y_n,Z\}$ be a canonical basis of $\mathfrak{k}_n$, so that $\mathrm{Span}(Z)=\mathfrak{z}$. Recall that the Lie brackets in $\mathfrak{k}_n$ are given by $[X_i,X_j]=[Y_i,Y_j]=0$ for all $i,j$, and 
$$[X_i,Y_j]= \left\{ \begin{array}{rl} Z & \text{if }  i=j; \\
                                    0 & \text{if }  i\neq j. \end{array} \right.$$

We will prove Theorem \ref{heisenberg} by induction on $d$, the number of eigenvalues of the derivations. To do this we need the base cases $d=2$ and $d=3$. The first case will be solved in a slightly more general context.

In the Lemmas \ref{twoeigenvalues}, \ref{derivations} and \ref{derivations2}, we will consider the class of nilpotent Lie algebras of the form $\mathfrak{k}_n\oplus \R^p$. This class consists of all $2$-nilpotent Lie algebras with derived algebra of dimension at most one. Denote this class by $\mathcal{C}$.

Let $\mathfrak{n}$ be a non-abelian algebra of $\mathcal{C}$. Observe that if $\mathfrak{k}$ is a sub-algebra of $\mathfrak{n}$ that does not contain the derived algebra $\mathfrak{n}'=[\mathfrak{n},\mathfrak{n}]$, then the normalizer of $\mathfrak{k}$ coincides with its centralizer. Moreover, if $\mathfrak{n}=V\oplus \mathfrak{n}'$ and $\mathfrak{k}\subset V$, then
\begin{equation}\label{intersection}\mathfrak{N}(\mathfrak{k})=\left(\mathfrak{N}(\mathfrak{k})\cap V\right)\oplus \mathfrak{n}'.\end{equation}
If $\mathfrak{n}_1$ and $\mathfrak{n}_2$ are Lie algebras in $\mathcal{C}$, and $\alpha$ and $\beta$ are derivations of $\mathfrak{n}_1$ and $\mathfrak{n}_2$ respectively, with the same positive eigenvalues $\lambda_1<\ldots<\lambda_d$, we write
\begin{equation}\label{decomposition}\mathfrak{n}_1=V_1\oplus \cdots\oplus V_d\text{ and }\mathfrak{n}_2=W_1\oplus \cdots\oplus W_d,\end{equation}
the direct-sum descomposition into the generalized eigenspaces of $\alpha$ and $\beta$ respectively.

\begin{lemma}\label{twoeigenvalues} Let $N_1$ and $N_2$ be two arbitrary groups with Lie algebras in $\mathcal{C}$. Suppose that there exists a bi-Lipschitz homeomorphism $\Phi:(N_1,\varrho_{\alpha})\to(N_2,\varrho_{\beta})$, where $\alpha$ and $\beta$ have two eigenvalues and $\dim V_2=\dim W_2=1$. Then $\alpha$ and $\beta$ have the same Jordan form. 
\end{lemma}

We will first prove some useful properties of the derivations of the algebras corresponding to the class $\mathcal{C}$. Let $\mathfrak{n}$ in $\mathcal{C}$ and $\alpha$ as in (\ref{decomposition}). We fix in the sequel a Jordan basis
$$\{X_{ir}^k:1\leq i \leq d, 1\leq r\leq n_i, 1\leq k \leq m_{ir}\}$$ 
of $\mathfrak{n}$ associated to $\alpha$ where $n_i$ is the number of Jordan blocks corresponding to the eigenspace $V_i$. $X_{ir}^1$ are the eigenvectors and $\alpha(X_{ir}^k)=\lambda_i X_{ir}^k+X_{ir}^{k-1}$ when $k>1$.

Notice that the derived algebra $\mathfrak{n}'$ is $\alpha$-invariant, and if it has dimension one, it is spanned by an eigenvector of $\alpha$ of eigenvalue $\lambda_p$ for some $p$.

\begin{lemma}\label{derivations} If $\left[X_{ir}^1,X_{js}^l\right]\neq 0$ then $l=m_{js}$. In particular, if two eigenvectors satisfy $\left[X_{ir}^1,X_{js}^1\right]\neq 0$, then $m_{ir}=m_{js}=1$.
\end{lemma}

\begin{proof}
If $m_{js}=1$ there is nothing to prove, so we may assume that $m_{js}>1$. Since $\left[X_{ir}^1,X_{js}^l\right]\neq 0$, by (\ref{ursula}) $\lambda_i+\lambda_j=\lambda_p$. For $l>1$, we have
\begin{align*}
\lambda_p \left[X_{ir}^1,X_{js}^l\right] &= \alpha \left[X_{ir}^1,X_{js}^l\right] = \left[\alpha\,X_{ir}^1,X_{js}^l\right]+\left[X_{ir}^1,\alpha\,X_{js}^l\right]\\
&=\lambda_i\left[X_{ir}^1,X_{js}^l\right]+\lambda_j\left[X_{ir}^1,X_{js}^l\right]+\left[X_{ir}^1,X_{js}^{l-1}\right]\\
&=\lambda_p \left[X_{ir}^1,X_{js}^l\right]+\left[X_{ir}^1,X_{js}^{l-1}\right].
\end{align*}
Then $\left[X_{ir}^1,X_{js}^l\right]=0$ for all $l<m_{js}$.
\end{proof}

\begin{lemma}\label{derivations2}
If $\left[X_{ir}^1,X_{js}^{m_{js}}\right]\neq 0$ and $m_{ir}\geq m_{js}$, then $m_{js}=m_{ir}$.
\end{lemma}
\begin{proof}
If $m_{js}=1$, the conclusion follows from Lemma \ref{derivations}. So we may assume that $m_{js}>1$. Let $k,l>1$ and suppose that $\left[X_{ir}^1,X_{js}^{m_{js}}\right]\neq 0$. As before, let $\lambda_i+\lambda_j=\lambda_p$. Then
\begin{align*}
\lambda_p \left[X_{ir}^k,X_{js}^l\right]&= \alpha \left[X_{ir}^k,X_{js}^l\right] = \left[\alpha\,X_{ir}^k,X_{js}^l\right]+\left[X_{ir}^k,\alpha\,X_{js}^l\right]\\
&=\lambda_i\left[X_{ir}^k,X_{js}^l\right]+\lambda_j\left[X_{ir}^k,X_{js}^l\right]+\left[X_{ir}^k,X_{js}^{l-1}\right]+\left[X_{ir}^{k-1},X_{js}^l\right]\\
&=\lambda_p \left[X_{ir}^k,X_{js}^l\right]+\left[X_{ir}^k,X_{js}^{l-1}\right]+\left[X_{ir}^{k-1},X_{js}^l\right].
\end{align*}
This implies that $\left[X_{ir}^{k-1},X_{js}^l\right]=-\left[X_{ir}^k,X_{js}^{l-1}\right]$. If $m_{ir}\geq m_{js}$, we have 
$$0\neq\left[X_{ir}^1,X_{js}^{m_{js}}\right]=-\left[X_{ir}^2,X_{js}^{m_{js}-1}\right]=\cdots=(-1)^{m_{js}-1}\left[X_{ir}^{m_{js}},X_{js}^1\right].$$
By Lemma \ref{derivations}, this implies that $m_{ir}=m_{js}$. 
\end{proof} 

\begin{proof}[Proof of Lemma \ref{twoeigenvalues}]
By \cite[Theorem 1.4]{Carrasco14}, the maximal size $m_1$ of the $\lambda_1$-blocks is a quasi-isometry invariant; i.e.\ $m_1^\alpha=m_1^\beta=m$. We will argue by induction on $m$.

If $m=1$, then the derivations $\alpha$ and $\beta$ are diagonalizable. This case follows clearly since $\alpha$ and $\beta$ both have $\mathrm{diag}(\lambda_1,\ldots,\lambda_1,\lambda_2)$ as diagonal form.

Suppose that $m>1$, and that the lemma is true whenever the maximal size of the $\lambda_1$-blocks is smaller than $m$. As we said above, the derived algebras $\mathfrak{n}'_1$ and $\mathfrak{n}'_2$ (if are not trivial) are contained in eigenspaces, which in this case must be $V_2$ and $W_2$.

Notice that Lemmas \ref{derivations} and \ref{derivations2} imply
$$\ker(\alpha-Id)^{m-1} \subset \mathfrak{N}(\mathfrak{h}_{\alpha})\cap V_1 \subset V_1.$$
Then we can choose $\{X_{ir}^k\}_{i,r,k}$ a Jordan basis of $\mathfrak{n}_1$ for $\alpha$, where $X_{11}^m,X_{12}^m,\cdots,X_{1q_{\beta}}^m$ belong to $\mathfrak{N}(\mathfrak{h}_{\alpha})$ and 
$$\mathrm{Span}\{X_{11}^m,X_{12}^m,\cdots,X_{1q_{\alpha}}^m\}\oplus \ker\left((\alpha-Id)^{m-1}\right) = \mathfrak{N}(\mathfrak{h}_{\alpha})\cap V_1.$$
In the same way we can choose a Jordan basis $\{Y_{ir}^k\}_{i,r,k}$ of $\mathfrak{n}_2$ for $\beta$. We have that
\begin{align*}
q_\alpha&=\dim\left(\mathfrak{N}(\mathfrak{h}_\alpha)\cap V_1\right)-\dim\mathrm{ker}(\alpha-Id)^{m-1};\\
q_\beta&=\dim\left(\mathfrak{N}(\mathfrak{h}_\beta)\cap W_1\right)-\dim\mathrm{ker}(\beta-Id)^{m-1}.
\end{align*}
By Lemmas \ref{derivations} and \ref{derivations2}, the Lie algebras $\mathfrak{h}_\alpha$ and $\mathfrak{h}_\beta$ are abelian and their dimensions are given by
\begin{align*}
\dim\mathfrak{h}_\alpha&=\text{number of maximal size $\lambda_1$-subblocks of }\alpha;\\
\dim\mathfrak{h}_\beta&=\text{number of maximal size $\lambda_1$-subblocks of }\beta.
\end{align*}
After composing with a left translation, we may suppose that $\Phi$ maps $H_{\alpha}$ into $H_{\beta}$, so $\dim(\mathfrak{h}_{\alpha})=\dim(\mathfrak{h}_{\beta})=a$ and
$$\dim \ker (\alpha-Id)^{m-1}=\dim \ker (\beta-Id)^{m-1} =\dim V_1-a.$$
By Lemma \ref{distancecosets}, $\Phi$ maps $N(H_{\alpha})$ into $N(H_{\beta})$ and since 
$$\mathfrak{N}(\mathfrak{h}_{\alpha})=\left(\mathfrak{N}(\mathfrak{h}_{\alpha})\cap V_1\right)\oplus V_2\text{ and }\mathfrak{N}(\mathfrak{h}_{\beta})=\left(\mathfrak{N}(\mathfrak{h}_{\beta})\cap W_1\right)\oplus W_2,$$
we have $\dim\left(\mathfrak{N}(\mathfrak{h}_{\alpha})\cap V_1\right)=\dim\left(\mathfrak{N}(\mathfrak{h}_{\beta})\cap W_1\right)$ and therefore $q_\alpha=q_\beta=q$.

The map $\Phi$ induces a bi-Lipschitz map between $N(H_{\alpha})/H_{\alpha}$ and $N(H_{\beta})/H_{\beta}$, where the quotients are equipped with the quasi-metrics $\tilde{\varrho}_{\alpha}$ and $\tilde{\varrho}_{\beta}$, induced by $\varrho_{\alpha}$ and $\varrho_{\beta}$ respectively. Observe that the quotient Lie algebras $\mathfrak{N}(\mathfrak{h}_{\alpha})/\mathfrak{h}_{\alpha}$ and $\mathfrak{N}(\mathfrak{h}_{\beta})/\mathfrak{h}_{\beta}$ are isomorphic to the sub-algebras
\begin{align*}
\mathfrak{m}_1&=\mathrm{Span}\{X_{ir}^k: (i,r,k)\neq (1,s,m)\text{ with } s>q \text{ and } (i,r,k)\neq (1,s,1)\text{ with } m_{1s}=m\};\\
\mathfrak{m}_2&=\mathrm{Span}\{Y_{ir}^k: (i,r,k)\neq (1,s,m)\text{ with } s>q \text{ and } (i,r,k)\neq (1,s,1)\text{ with } m_{1s}=m\}
\end{align*}
respectively. Since $\tilde{\alpha}=\alpha|_{\mathfrak{m}_1}$ and $\tilde{\beta}=\beta|_{\mathfrak{m}_2}$ correspond with the induced derivations on the quotients by $\alpha$ and $\beta$ respectively we have that the quasi-metrics $\tilde{\varrho}_{\alpha}$ and $\tilde{\varrho}_{\beta}$ are bi-Lipschitz equivalent with $\varrho_{\tilde{\alpha}}$ and $\varrho_{\tilde{\beta}}$. Notice that the Jordan blocks of these derivations have size at most $m-1$.

Now we have a bi-Lipschitz homeomorphism $\tilde{\Phi}:(N(H_{\alpha})/H_{\alpha},\varrho_{\tilde{\alpha}}) \to (N(H_{\beta})/H_{\beta},\varrho_{\tilde{\beta}})$ where the groups $N(H_{\alpha})/H_{\alpha}$ and $N(H_{\beta})/H_{\beta}$ have their Lie algebras in $\mathcal{C}$. We use the induction hypothesis to claim that $\tilde{\alpha}$ and $\tilde{\beta}$ have the same Jordan form. 

We can recover the Jordan form of $\alpha$ from that of $\tilde{\alpha}$. This is done by choosing $q$ $\lambda_1$-subblocks from $\tilde{\alpha}$ of size $m-1$ and adding to them one column, and choosing $a-q$ $\lambda_1$-subblocks of size $m-2$ and adding two columns. The same holds for $\beta$. This shows that $\alpha$ and $\beta$ have the same Jordan form and completes the proof.
\end{proof}

\begin{lemma}\label{threeeigenvalues} 
Suppose that $d=3$. Let $\Phi:(K_n,\varrho_{\alpha})\to (K_n,\varrho_{\alpha})$ be a bi-Lipschitz homeomorphism. Then $\alpha$ and $\beta$ have the same Jordan form.
\end{lemma}

\begin{proof} By Lemma \ref{dimclass}, $\Phi$ maps the left cosets of $U_{\alpha}$ into the left cosets of $U_{\beta}$. After composing $\Phi$ with a left translation we may assume that $\Phi(U_{\alpha})=U_{\beta}$. So we have a bi-Lipschitz map between $(U_{\alpha},\varrho_{\alpha|_{V_1}})$ and $(U_{\beta},\varrho_{\beta|_{W_1}})$. Since these groups are abelian, by \cite{XiangdongXie14} we have that $\alpha|_{V_1}$ and $\beta|_{W_1}$ have the same Jordan form.

It is easy to see that $\mathfrak{N}(V_1)=V_1\oplus V_3$ and $\mathfrak{N}(W_1)=W_1\oplus W_3$. Then $\mathfrak{k}_n/\mathfrak{N}(V_1)$ and $\mathfrak{k}_n/\mathfrak{N}(W_1)$ are isomorphic to $V_2$ and $W_2$ respectively, and are therefore abelian. The quasi-metrics induced on these quotients by $\varrho_{\alpha}$ and $\varrho_{\beta}$ are bi-Lipschitz equivalent with $\varrho_{\alpha|_{V_2}}$ and $\varrho_{\beta|_{W_2}}$. Since $\Phi$ induces a bi-Lipschitz homeomorphism 
$$\tilde{\Phi}:(K_n/N(U_{\alpha}),\varrho_{\alpha|_{V_2}}) \to (K_n/N(U_{\beta}),\varrho_{\beta|_{W_2}}),$$
we conclude that $\alpha|_{V_2}$ and $\beta|_{W_2}$ have the same Jordan form. This implies that $\alpha$ and $\beta$ have the same Jordan form. 
\end{proof}

Let $\alpha$ and $\beta$ be as in the statement of Theorem \ref{heisenberg}. By Theorem \ref{maintheorem}, their eigenvalues are the same up to a scalar multiple, so we may assume that they have the same eigenvalues. We will write $V_1,\ldots,V_d$ and $W_1,\ldots,W_d$ the generalized eigenspaces associated to $\alpha$ and $\beta$ respectively. By (Fact 4), the quasi-isometry induces a bi-Lipschitz homeomorphism $\Phi:(K_n,\varrho_{\alpha}) \to (K_n,\varrho_{\beta})$. We will prove by induction on $d$ that this implies that $\alpha$ and $\beta$ have the same Jordan form.
 
Lemmas \ref{twoeigenvalues} and \ref{threeeigenvalues} are the base cases of the induction. Suppose that the claim is true for $d'<d$ and $d> 3$. By Lemma \ref{dimclass}, we may assume that $\Phi$ sends $U_{\alpha}$ into $U_{\beta}$. As in \ref{threeeigenvalues}, we know that $\alpha|_{V_1}$ and $\beta|_{W_1}$ have the same Jordan form. 

Observe that $\mathfrak{N}(V_1)=\bigoplus_{i\neq d-1}V_i$ and $\mathfrak{N}(W_1)=\bigoplus_{i\neq d-1}W_i$. Moreover, we identify the quotients
$$\mathfrak{N}(V_1)/V_1=\bigoplus_{i\neq 1,d-1}V_i\text{ and }\mathfrak{N}(W_1)/W_1=\bigoplus_{i\neq 1,d-1}W_i.$$
We have that $\Phi$ maps $N(U_{\alpha})$ into $N(U_{\beta})$, and therefore induces a bi-Lipschitz map 
$$\tilde{\Phi}:N(U_{\alpha})/U_{\alpha}, \to N(U_{\beta})/U_{\beta},$$
where the quasi-metrics are the induced by $\alpha|_{\mathfrak{N}(V_1)/V_1}$ and $\beta|_{\mathfrak{N}(W_1)/W_1}$ respectively. These quotient groups are Heisenberg groups with less eigenvalues than $d$. Then $\alpha|_{N(V_1)/V_1}$ and $\beta|_{N(W_1)/W_1}$ have the same Jordan form.

Finally, using that $\Phi$ induces a bi-Lipschitz homeomorphism between the abelian groups $K_n/N(U_{\alpha})$ and $K_n/N(U_{\beta})$, by \cite{XiangdongXie14}, we have that $\alpha|_{V_{d-1}}$ and $\beta|_{W_{d-1}}$ have the same Jordan form. This completes the proof.

\subsection*{Acknowledgments} This work was supported by a research grant of the CAP (Comisi\'on Acad\'emiaca de Posgrado) of the Universidad de la Rep\'ublica. The authors are very grateful.

\Addresses


\begin{thebibliography}{{Cor}12}

\bibitem[BS00]{BonkSchramm00}
M.~Bonk and O.~Schramm.
\newblock Embeddings of {G}romov hyperbolic spaces.
\newblock {\em Geom. Funct. Anal.}, 10(2):266--306, 2000.

\bibitem[{Car}14]{Carrasco14}
M.~{Carrasco Piaggio}.
\newblock {Orlicz spaces and the large scale geometry of Heintze groups}.
\newblock {\em ArXiv e-prints}, November 2014.

\bibitem[{Cor}12]{Cornulier13}
Y.~{Cornulier}.
\newblock {On the quasi-isometric classification of focal hyperbolic groups}.
\newblock {\em ArXiv e-prints}, December 2012.

\bibitem[Ham87]{UrsulaHamenstadt87}
Ursula Hamenst{\"a}dt.
\newblock {\em Zur {T}heorie von {C}arnot-{C}arath\'eodory {M}etriken und ihren
  {A}nwendungen}.
\newblock Bonner Mathematische Schriften [Bonn Mathematical Publications], 180.
  Universit\"at Bonn, Mathematisches Institut, Bonn, 1987.
\newblock Dissertation, Rheinische Friedrich-Wilhelms-Universit{\"a}t Bonn,
  Bonn, 1986.

\bibitem[Ham89]{Hamenstadt89}
Ursula Hamenst{\"a}dt.
\newblock A new description of the {B}owen-{M}argulis measure.
\newblock {\em Ergodic Theory Dynam. Systems}, 9(3):455--464, 1989.

\bibitem[Hei74]{Heintze74}
Ernst Heintze.
\newblock On homogeneous manifolds of negative curvature.
\newblock {\em Math. Ann.}, 211:23--34, 1974.

\bibitem[HP97]{Hersonsky-Paulin97}
Sa'ar Hersonsky and Fr{\'e}d{\'e}ric Paulin.
\newblock On the rigidity of discrete isometry groups of negatively curved
  spaces.
\newblock {\em Comment. Math. Helv.}, 72(3):349--388, 1997.

\bibitem[LX15]{LeDonneXie14}
E.~{Le Donne} and X.~{Xie}.
\newblock {Rigidity of fiber-preserving quasisymmetric maps}.
\newblock {\em ArXiv e-prints}, January 2015.

\bibitem[Mai15]{Mainkar15}
Meera~G. Mainkar.
\newblock Graphs and two-step nilpotent {L}ie algebras.
\newblock {\em Groups Geom. Dyn.}, 9(1):55--65, 2015.

\bibitem[Pan89a]{Pansu89t}
Pierre Pansu.
\newblock Cohomologie {$L^p$} des vari\'et\'es \`a courbure n\'egative, cas du
  degr\'e {$1$}.
\newblock {\em Rend. Sem. Mat. Univ. Politec. Torino}, (Special Issue):95--120
  (1990), 1989.
\newblock Conference on Partial Differential Equations and Geometry (Torino,
  1988).

\bibitem[Pan89b]{PierrePansu89d}
Pierre Pansu.
\newblock Dimension conforme et sph\`ere \`a l'infini des vari\'et\'es \`a
  courbure n\'egative.
\newblock {\em Ann. Acad. Sci. Fenn. Ser. A I Math.}, 14(2):177--212, 1989.

\bibitem[Pan89c]{PierrePansu89m}
Pierre Pansu.
\newblock M\'etriques de {C}arnot-{C}arath\'eodory et quasiisom\'etries des
  espaces sym\'etriques de rang un.
\newblock {\em Ann. of Math. (2)}, 129(1):1--60, 1989.

\bibitem[Pau96]{FredericPaulin96}
Fr{\'e}d{\'e}ric Paulin.
\newblock Un groupe hyperbolique est d\'etermin\'e par son bord.
\newblock {\em J. London Math. Soc. (2)}, 54(1):50--74, 1996.

\bibitem[SX12]{NageswariXie12}
Nageswari Shanmugalingam and Xiangdong Xie.
\newblock A rigidity property of some negatively curved solvable {L}ie groups.
\newblock {\em Comment. Math. Helv.}, 87(4):805--823, 2012.

\bibitem[Xie12]{XiangdongXie12}
Xiangdong Xie.
\newblock Quasisymmetric maps on the boundary of a negatively curved solvable
  {L}ie group.
\newblock {\em Math. Ann.}, 353(3):727--746, 2012.

\bibitem[Xie14]{XiangdongXie14}
Xiangdong Xie.
\newblock Large scale geometry of negatively curved
  {$\Bbb{R}^n\rtimes\Bbb{R}$}.
\newblock {\em Geom. Topol.}, 18(2):831--872, 2014.

\bibitem[Xie15a]{XiangdongXie14c}
Xiangdong Xie.
\newblock Quasiisometries of negatively curved homogeneous manifolds associated
  with {H}eisenberg groups.
\newblock {\em J. Topol.}, 8(1):247--266, 2015.

\bibitem[Xie15b]{XiangdongXie14b}
Xiangdong Xie.
\newblock Rigidity of quasi-isometries of {HMN} associated with
  non-diagonalizable derivation of the {H}eisenberg algebra.
\newblock {\em Q. J. Math.}, 66(1):353--367, 2015.

\end{thebibliography}
\end{document}